\documentclass[11pt,letterpaper]{article}
\usepackage[margin=1in]{geometry}
\usepackage{amsmath,amssymb,amsthm}
\usepackage{mathrsfs}
\usepackage{tikz}
\usepackage{xcolor}
\usepackage{needspace}

\usepackage[colorlinks=true,linkcolor=.,citecolor=.,urlcolor=.]{hyperref}
\usepackage[numbers]{natbib}  

\hypersetup{%
  bookmarksnumbered=true,%
  bookmarks=true,%
  colorlinks=true,%
  linkcolor=blue,%
  citecolor=blue,%
  filecolor=blue,%
  menucolor=blue,%
  pagecolor=blue,%
  urlcolor=blue,%
  pdfnewwindow=true,%
  pdfstartview=FitBH}

\newtheorem{theorem}{Theorem}[section]
\newtheorem{definition}[theorem]{Definition}
\newtheorem{lemma}[theorem]{Lemma}
\newtheorem{corollary}[theorem]{Corollary}
\newtheorem{remark}[theorem]{Remark}
\newcommand{\G}{\mathcal G}
\newcommand{\V}{\mathcal V}
\newcommand{\E}{\mathcal E}
\newcommand{\FSF}{\mathrm{FSF}}
\newcommand{\WSF}{\mathrm{WSF}}
\newcommand{\PP}{\mathbb P}
\newcommand{\EE}{\mathbb E}
\newcommand{\cyc}{\Diamond}

\newcommand{\iF}{i^e_{\mathrm F}}
\title{A zero--one law for swap-invariant events in uniform spanning forests}
\author{Xinyi Li\thanks{Peking University, Beijing, China. \href{mailto:xinyili@bicmr.pku.edu.cn}{xinyili@bicmr.pku.edu.cn}} \qquad Yao Yu\thanks{The University of Chicago, Chicago, USA. \href{mailto:yaoyu@uchicago.edu}{yaoyu@uchicago.edu}}}
\date{\today}
\begin{document}
\maketitle
\begin{abstract}
We prove that every event invariant under swaps has probability zero or one for either the free or the wired spanning forest on any connected, locally finite, infinite network. Applications include the constancy of the number of components and of the total number of excessive ends. For the free forest, these results resolve longstanding questions of Benjamini, Lyons, Peres, and Schramm \cite[Question~15.7]{BLPS} on constancy of the component count and of Lyons, Morris, and Schramm \cite[Question~7.7]{LMS} on the zero--one law for one-endedness. In the latter case, we obtain the stronger conclusion that the total number of excessive ends is deterministic.
\end{abstract}

\section{Introduction}
Let $\G=(\V,\E,\mathfrak{c})$ be a connected, locally finite, infinite network, with conductances $\mathfrak{c}:\E\to(0,\infty)$. For a finite connected subgraph, the weighted spanning tree measure assigns to each spanning tree a probability proportional to the product of its edge conductances; when $\mathfrak{c}\equiv1$, this is the uniform spanning tree measure. Along any increasing exhaustion of $\G$ by finite connected subgraphs, these measures converge weakly to the \textbf{free spanning forest} measure, denoted by FSF. Identifying the exterior of each subgraph to one extra vertex before taking the limit gives the \textbf{wired spanning forest} measure, denoted by WSF. Both limits are independent of the exhaustion. 

The number of components has been central since Pemantle \cite{Pemantle}. On $\mathbb Z^d$, he proved that the free and wired measures agree and that the resulting forest is connected a.s.\ for $d\leq4$, whereas it has infinitely many components a.s.\ for $d\geq5$. Benjamini, Lyons, Peres, and Schramm \cite{BLPS} developed the theory on general networks. Their Theorem~9.4 identifies the number of WSF components via non-intersection probabilities for independent random walks; in particular, the number is a.s.\ constant. They asked whether the same constancy holds for the FSF on every infinite network \cite[Question~15.7]{BLPS}. The question also appears as Question~10.28 in Lyons and Peres \cite{LP}.

Previously, the FSF component-count question was known only in special cases. If the free and wired measures agree, constancy follows from the WSF theorem. If the automorphism group of the network has an infinite orbit, it follows from ergodicity \cite[Remark~9.8]{BLPS}. For a general network, tail triviality gives the weaker conclusion that the number of FSF components is either a.s.\ finite or a.s.\ infinite \cite[Remark~9.8]{BLPS}. Related work establishes stronger conclusions under additional hypotheses: Hutchcroft and Nachmias \cite{HN}, and independently Tim\'ar \cite{Timar}, proved the one-or-infinitely-many alternative for the FSF on unimodular transitive graphs, and Hutchcroft and Nachmias proved it more generally on unimodular random rooted networks with finite expected conductance at the root \cite[Theorem~1.10]{HN}; Hutchcroft and Nachmias \cite{HNplanar} also proved connectivity on bounded-degree proper plane graphs with edge conductances bounded above. 

A second longstanding question concerns the ends of the FSF. Lyons, Morris, and Schramm \cite[Question~7.7]{LMS} asked whether the event that every component is one-ended has probability zero or one, for both the WSF and the FSF. The same question is recorded as Question~10.59 in Lyons and Peres \cite{LP}. Question~7.8 of Lyons, Morris, and Schramm asks about constancy of the total number of excessive ends for the WSF, which implies the wired case of Question~7.7; Hutchcroft \cite[Theorem~2.2]{Hutchcroft} settled Question~7.8 using the interlacement Aldous--Broder construction.

In this work, we prove a zero--one law for all events whose occurrence is unchanged whenever one present edge is removed and one absent edge is added. 
Applied to the component count, this answers Question~15.7 of Benjamini, Lyons, Peres, and Schramm. Applied to the total number of excessive ends, i.e.\ the sum over components of their numbers of ends minus one, it answers the FSF case of Question~7.7 of Lyons, Morris, and Schramm and strengthens its conclusion. The same argument recovers the corresponding WSF results. It also gives zero--one laws for the events that every component is recurrent or that every component has finitely many ends; see Section~\ref{sec:applications}. Throughout, every infinite number of ends is recorded as $\infty$.

\medskip

We now state the theorem precisely. 
\begin{definition}
A \emph{swap} replaces one edge of a spanning forest by one previously absent edge. 
An event $A$ is \textbf{swap-invariant} if
\[
\mathbf1_A(F)=\mathbf1_A(F')
\quad\text{for any spanning forests $F$ and $F'$ with }|F\setminus F'|=|F'\setminus F|=1.
\]
\end{definition}

\begin{theorem}\label{thm:main}
For a connected, locally finite, infinite network $\G$, let $\mu$ be its FSF (resp.\ WSF) measure. Every swap-invariant event $A$ satisfies $\mu(A)\in\{0,1\}$. 
\end{theorem}

Hutchcroft \cite[Conjecture~7.4]{Hutchcroft} proposed a zero--one law for update-stable events of either forest. The updates considered there add a specified absent edge and remove a suitably chosen present edge, so that the updated forest has a law absolutely continuous with respect to the original forest law. Update stability means that each such random update preserves the event a.s. Our theorem makes no absolute-continuity assumption on an update rule: it requires the event to be unchanged under every admissible swap. This implies stability under Hutchcroft's updates.  The converse is not established here, so we do not claim the full conjecture. Since every tail event is swap-invariant, our theorem also contains the tail triviality of both forests \cite[Theorem~8.3]{BLPS}.

\paragraph{Sketch of proof strategy.}
The starting observation is that, for an edge with inclusion probability strictly between zero and one, the conditional forests obtained by requiring it present or absent can be coupled to differ by one deletion and one insertion. Every swap-invariant event therefore has the same probability under both laws. The argument combines finite spanning-tree edge exchange, the transfer-current representation of Burton and Pemantle \cite{BP} and its forest version \cite[Theorem~10.15]{LP}, and the codimension-one coupling for projection determinantal probability measures (PDPMs) \cite[Proposition~10.3]{Lyons}. We use contraction and deletion for both the FSF and the WSF. 
The same argument applies after finitely many edge states are revealed. In addition, edges with conditional inclusion probability zero or one reveal no further information. The edge-revelation martingale is thus constant, and its limit is the event indicator.

\paragraph{Acknowledgments.} The authors thank Shiping Cao, Ewain Gwynne, and Yuval Peres for inspiring discussions. XL is supported by Beijing Natural Science Foundation (JQ26001) and the National Key R\&D Program of China (No.~2021YFA1002700). YY is grateful to the REU program at the University of Chicago, where she first learned about this problem.
\paragraph{AI Usage Disclosure.} The development of the key idea in this work benefited from  discussions with commercial large language models. The authors take full responsibility for the correctness of this work.

\section{Proof of the main theorem}
We first recall the projection determinantal representation and the coupling result used in the proof. We then compare the conditional forests obtained by requiring a fixed edge to be present or absent, extend the comparison to finite conditioning, and finish with the martingale argument.

\begin{definition}
Let $I$ be an at most countable index set. A probability measure $\PP$ on random subsets $A\subseteq I$ is called a determinantal probability measure (DPM) if there is a kernel $Q$ such that
\[
\PP[B\subseteq A]=\det(Q_B)\qquad\text{for every finite }B\subseteq I,
\]
where $Q_B$ is the submatrix with rows and columns indexed by $B$.

Given positive weights $w:I\to(0,\infty)$, the Hilbert space $\ell^2(I,w^{-1})$ is defined with inner product 
\[
\langle f,g\rangle=\sum_{x\in I}w^{-1}(x)f(x)g(x).
\]
For any closed $H\subseteq\ell^2(I,w^{-1})$, the \textbf{projection determinantal probability measure} (PDPM) $\PP^H$ is defined as the DPM with orthogonal projection kernel $Q_H(x,y)=(P_H\mathbf1_x)(y)$, where $\mathbf1_x$ is the coordinate vector at $x$. 
For a network, we take $I$ to be its edge set and $w$ its conductances. 
\end{definition}

PDPMs admit the following coupling. 
\begin{theorem}[Proposition 10.3 of \cite{Lyons}]\label{thm:coupling}
Let $I$ be at most countable and $w:I\to(0,\infty)$. Let $H_1\subset H_2$ be closed subspaces of $\ell^2(I,w^{-1})$ with $\dim(H_2/H_1)=1$. There is a coupling $(A_1,A_2)$ of $\PP^{H_1}$ and $\PP^{H_2}$ such that
$A_1\subset A_2\ \text{and}\ |A_2\setminus A_1|=1\ \text{a.s}.$
\end{theorem}

The next theorem identifies the FSF as a PDPM.
\begin{theorem}[Theorem 10.15 and Section 9.1 of \cite{LP}]\label{thm:FSF}
Let $G=(V,E,c)$ be a network, possibly with multiple edges and self loops. The FSF of $G$ is a DPM with kernel $Q=(\iF(f))_{e,f\in E}$, where $\iF$ is the unit free current from the tail of $e$ to its head. Let $\cyc(G)$ be the closed span of the unit flows around finite cycles\footnote{The unit flow of a cycle is $+1$ on edges traversed along their chosen orientation, $-1$ otherwise, and $0$ off the cycle. } in $\ell^2(E,c^{-1})$. Since $\iF:=P_{\cyc(G)^\perp}\mathbf1_e$, the FSF has law $\PP^{\cyc(G)^\perp}$, thus a PDPM.
\end{theorem}

\begin{remark}
See \cite[Section~9.1]{LP} for the definitions of free current and $\Diamond^\perp(G)$. Usually currents are represented
by antisymmetric functions on oriented edges, but we could choose an orientation for each edge and identify such functions with functions on the unoriented edge set. Multiple edges are kept distinct. Self-loops are included in the cycle space, but are absent from any spanning forest.
\end{remark}

Let $\PP$ stand for the FSF measure. Fix an edge $e$ such that $\PP[e\in\FSF]\in(0,1)$. Let $\mathcal S\sim \PP[\;\cdot\mid e\in\FSF]$, and $\mathcal T \sim \PP[\;\cdot\mid e\notin\FSF]$. These conditional laws are the weak limits of the corresponding conditional spanning-tree measures along a finite exhaustion. Indeed, both conditioning events are cylinder events with positive limiting probability.

Write $\G_e^+$ for the network obtained by contracting $e$, and $\G_e^-$ for the network obtained by deleting $e$, with all other conductances unchanged. The finite contraction/deletion identities, followed by passage to the limit, show that $\mathcal S\setminus\{e\}$ has the FSF law on $\G_e^+$, while $\mathcal T$ has the FSF law on $\G_e^-$. Note that for $\G^+_e$, multiple edges and self loops may arise after contraction. The graph $\G_e^-$ must be connected; otherwise $e$ would be a bridge and would belong to the FSF a.s.

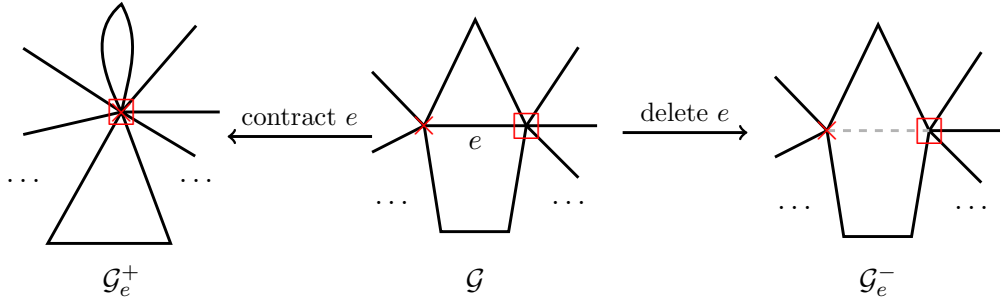
\begin{figure}[htbp]
  \centering
  \begin{tikzpicture}[x=1.3cm,y=1.3cm,line width=1.1pt]
  \begin{scope}[shift={(0,0)}]
    \draw (0,0)--(-1,0.65) (0,0)--(-1,-0.23)
      (0,0)--(0.76,0.88) (0,0)--(1,0) (0,0)--(0.75,-0.45);
    \draw (0,0) to[out=116,in=220] (0,1.1) to[out=-62,in=78] (0,0);
    \draw (0,0)--(-0.75,-1.34)--(0.5,-1.34)--cycle;
    \node[red,font=\Large] at (0,0) {$\times$};
     \node[red,font=\Large] at (0,0) {$\square$};
    \node at (-1,-0.7) {$\ldots$};
    \node at (0.76,-0.7) {$\ldots$};
    \node at (0,-1.72) {$\G_e^+$};
  \end{scope}

  \draw[->] (2.55,-0.25)--(1.08,-0.25)
    node[midway,above,font=\small] {contract $e$};

  \begin{scope}[shift={(3.6,-0.14)}]
    \draw (-0.52,0)--(0,1.08)--(0.52,0)--(0.34,-1.08)--(-0.35,-1.08)--cycle;
    \draw (-0.52,0)--(-1.05,0.55) (-0.52,0)--(-1.05,-0.27)
      (0.52,0)--(1.05,0.8) (0.52,0)--(1.24,0) (0.52,0)--(1.05,-0.53);
    \draw (-0.52,0)--(0.52,0) node[midway,below] {$e$};
    \node[red,font=\Large] at (-0.52,0) {$\times$};
    \node[red,font=\Large] at (0.52,0) {$\square$};
    \node at (-0.85,-0.78) {$\ldots$};
    \node at (0.95,-0.78) {$\ldots$};
    \node at (0,-1.58) {$\G$};
  \end{scope}

  \draw[->] (5.1,-0.21)--(6.37,-0.21)
    node[midway,above,font=\small] {delete $e$};

  \begin{scope}[shift={(7.7,-0.19)}]
    \draw (-0.52,0)--(0,1.08)--(0.52,0)--(0.34,-1.08)--(-0.35,-1.08)--cycle;
    \draw (-0.52,0)--(-1.05,0.55) (-0.52,0)--(-1.05,-0.27)
      (0.52,0)--(1.05,0.8) (0.52,0)--(1.24,0) (0.52,0)--(1.05,-0.53);
    \draw[gray!60,dashed] (-0.52,0)--(0.52,0);
    \node[red,font=\Large] at (-0.52,0) {$\times$};
    \node[red,font=\Large] at (0.52,0) {$\square$};
    \node at (-0.85,-0.78) {$\ldots$};
    \node at (0.95,-0.78) {$\ldots$};
    \node at (0,-1.53) {$\G_e^-$};
  \end{scope}
\end{tikzpicture}
  \caption{The original graph $\mathcal{G}$ and the altered graphs $\mathcal{G}^+_e$ and $\mathcal{G}^-_e$}
  \label{figure}
\end{figure}

\begin{lemma}\label{lem:swap}
For any $e$ such that $\PP[e\in\FSF]\in(0,1)$, $\mathcal S$ and $\mathcal T$ can be coupled such that $|\mathcal S\setminus\mathcal T|=1$ and $|\mathcal T\setminus\mathcal S|=1$ a.s.
\end{lemma}

\begin{proof}
Let $\mathcal S'$ and $\mathcal T'$ have the FSF laws on $\G_e^+$ and $\G_e^-$. It suffices to couple them so that $\mathcal S'\subset\mathcal T'$ and $|\mathcal T'\setminus\mathcal S'|=1$ a.s. We then set $\mathcal S=\mathcal S'\cup\{e\}$ and $\mathcal T=\mathcal T'$.

By Theorem~\ref{thm:FSF}, these laws correspond to the closed subspaces $H_1=\cyc^\perp(\G_e^+)$ and $H_2=\cyc^\perp(\G_e^-)$. Every finite cycle in $\G_e^-$ is a cycle flow in $\G_e^+$, so
\[
H_1=\cyc^\perp(\G_e^+)\subset\cyc^\perp(\G_e^-)=H_2
\subset\ell^2(\E\setminus\{e\},\mathfrak{c}^{-1}).
\]

Choose a finite simple path between the two endpoints of $e$ in $\G_e^-$, with edges $u_1,\ldots,u_n$. Write $\varepsilon_k=1$ or $-1$ according to whether the path traverses $u_k$ along or against its chosen orientation. For $i\in H_2$, set
\[
L(i)=\sum_{k=1}^n\varepsilon_k \mathfrak{c}^{-1}(u_k)i(u_k).
\]
Orthogonality to finite cycles makes this path integral independent of the chosen path. Equivalently, $i\in H_2$ is a $\mathfrak{c}$-gradient of a function on $\V$. Such a gradient belongs to $H_1$ precisely when its potential takes the same value at the two endpoints of $e$, i.e. $L(i)=0$. Thus $H_1=\ker(L|_{H_2})$. 

To see that this kernel has codimension one, let $j$ be the restriction of $\iF$ to $\E\setminus\{e\}$. We have $j\in H_2$, and orthogonality of $\iF$ to the cycle formed by the path and $e$ gives
\[
L(j)=\mathfrak{c}^{-1}(e)\iF(e)>0,
\]
since $\iF(e)=\PP[e\in\FSF]>0$. In particular, $j\notin H_1$. Moreover, for every $i\in H_2$,
\[
i-\frac{L(i)}{L(j)}\,j\in\ker(L|_{H_2})= H_1.
\]
Hence $H_1$ has codimension one in $H_2$. This is the single degree of freedom described above: the potential difference between the two endpoints of $e$.

Applying Theorem~\ref{thm:coupling} on $\E\setminus\{e\}$ now gives the required coupling of $\mathcal S'$ and $\mathcal T'$.
\end{proof}
\begin{remark}\label{rem:wsfswap}
The same proof applies to the WSF as well. WSF is the PDPM $\mathbb{P}^\star$ (see Section 9.1 and Theorem 10.5 from \cite{LP}), and the one-dimensional difference between the two $\star$ spaces can be spanned by the projection of the star current centered at one endpoint of $e$ in $\G_e^-$ onto the orthogonal complement of the smaller $\star$ space. 

An additional obstacle in WSF is that $\mathcal{G}^-_e$ may have more than one  connected component. We take WSF on such a graph to be the union of independent WSFs on those components, and the projection kernel is simply block diagonal.
This is also the limit obtained by wiring the finite approximations of the components to one common boundary vertex: a spanning tree of their union is exactly an independent choice of one spanning tree in each piece.
\end{remark}

\begin{corollary}\label{cor:conditioning}
Let $\mu$ be FSF or WSF, and let $A$ be a swap-invariant event. For every edge $e$ with $\mu(e\in F)\in(0,1)$,
\[
\mu(A\mid e\in F)=\mu(A\mid e\notin F)=\mu(A).
\]
\end{corollary}

\begin{proof}
For FSF use Lemma~\ref{lem:swap}, and for WSF use Remark~\ref{rem:wsfswap}. Both coupled configurations differ by a swap, so their indicators of $A$ agree. Averaging the probabilities gives the last equality.
\end{proof}

\begin{remark}\label{rem:finiteconditioning}
The same argument applies after finitely many edge states have been specified. To compare the two further conditional forests on the original graph, we undo the previous contractions and restore the edges already known to be present. Both forests then have exactly the prescribed revealed edges and differ by one swap, so their event indicators agree. 
\end{remark}

We now apply the corollary successively as the edges are revealed.
\begin{proof}[Proof of Theorem~\ref{thm:main}]
Fix a swap-invariant event $A$ and let $F$ have law $\mu$. List the edges as $e_1,e_2,\ldots$, and let $\mathcal F_n=\sigma(\mathbf1_{\{e_1\in F\}},\ldots,\mathbf1_{\{e_n\in F\}})$, with $\mathcal F_0$ trivial. Write $\EE$ for the respective expectation. Then the bounded martingale $\EE[\mathbf1_A(F)\mid\mathcal F_n]$ converges a.s.\ to $\mathbf1_A(F)$.

Fix an atom of $\mathcal F_n$. If the conditional probability that $e_{n+1}$ is present is $0$ or $1$, revealing it gives no further information on this atom. Otherwise, Corollary \ref{cor:conditioning} and Remark~\ref{rem:finiteconditioning} show that
\[
\EE[\mathbf1_A(F)\mid\mathcal F_{n+1}]=\EE[\mathbf1_A(F)\mid\mathcal F_n]
\qquad\text{a.s.\ for every }n\geq0.
\]
Passing to the limit gives
$\mu(A)
=\EE[\mathbf1_A(F)\mid\mathcal F_0]
=\lim_{n\to\infty}\EE[\mathbf1_A(F)\mid\mathcal F_n]
=\mathbf1_A(F)\in\{0,1\}
\ \text{a.s.}$
\end{proof}

\section{Applications}\label{sec:applications}
We start with a trivial extension of the main theorem. 
\begin{corollary}\label{cor: main}
    Let $\G$ be a connected, locally finite, infinite network, and let $\mu$ be its FSF (resp.\ WSF) measure. Every measurable swap-invariant statistic with values in a standard Borel space is $\mu$-a.s.
constant.
\end{corollary}
\paragraph{Number of components.} Our first application answers the component-count question of Benjamini, Lyons, Peres, and Schramm \cite[Question~15.7]{BLPS}, also recorded in \cite[Question~10.28]{LP}.
\begin{corollary}\label{cor:components}
The number of components in FSF or WSF is a.s.\ constant.
\end{corollary}
\begin{proof}
Let $\mathcal S,\mathcal T$ differ by a single swap, and let $N$ be the number of components of the spanning forest $\mathcal S\cap\mathcal T$. Each of $\mathcal S$ and $\mathcal T$ is obtained by adding one edge to this common forest. Since neither contains a cycle, each added edge joins two distinct components. If $N<\infty$, both forests therefore have $N-1$ components; if $N=\infty$, both have infinitely many components. Thus the component count is swap-invariant. It is measurable, and Corollary~\ref{cor: main} applies.
\end{proof}

Corollary~\ref{cor:components} establishes constancy but does not characterize the number of FSF components in terms of the underlying network. We pursue such a characterization in separate work \cite{LiYu}. We also remark that whether the FSF on every transitive network has either one or infinitely many components \cite[Question~15.6]{BLPS}, \cite[Question~10.29]{LP} remains open in the nonunimodular case; the unimodular case was settled by Hutchcroft and Nachmias \cite{HN} and independently by Tim\'ar \cite{Timar}.

\paragraph{Excessive ends.}
An \textbf{end} of a locally finite tree is an equivalence class of infinite simple rays, where two rays are equivalent if their tails lie in the same component after any finite vertex deletion. Write $\operatorname{ends}(T)\in\{1,2,\ldots,\infty\}$ for the number of ends of an infinite tree, recording any infinite cardinality as $\infty$. 
We call a spanning forest \textbf{essential} if all of its components are infinite. In particular, FSF and WSF are supported on such forests. For an essential spanning forest $F$, define
\[
\xi(F)=\sum_{T\in\operatorname{Comp}(F)}\bigl(\operatorname{ends}(T)-1\bigr)
\in\{0,1,2,\ldots,\infty\}.
\]

\begin{corollary}\label{cor:ends}
For both FSF and WSF, the total number of excessive ends $\xi(F)$ is a.s.\ constant. In particular, the event that every component is one-ended has probability zero or one.
\end{corollary}
\begin{proof}
Suppose that $F,F'$ are essential and differ by one swap, and put $C=F\cap F'$. Deleting one edge from an infinite tree creates at most one finite component. If $C$ has one, each of the two inserted edges must join that finite component to an infinite one, since both resulting forests are essential. Attaching a finite tree does not change the ends of the infinite component. Thus $\xi(F)$ and $\xi(F')$ both equal the sum of the excesses of the infinite components of $C$.
If $C$ has no finite component, each insertion joins two infinite trees. The set of ends of the resulting tree is the disjoint union of their sets of ends. If the sum of excesses in $C$ is finite, each insertion increases it by exactly one; if it is infinite, it remains infinite. Hence $\xi(F)=\xi(F')$ in this case as well.
The statistic is measurable, so Corollary~\ref{cor: main} applies. The final assertion is the case $\xi(F)=0$.
\end{proof}

Corollary~\ref{cor:ends} answers the remaining FSF case of Question~10.59 of Lyons and Peres \cite{LP}: the event that every component is one-ended is precisely $\{\xi(F)=0\}$ and hence has probability zero or one. The WSF case was already settled by Hutchcroft \cite[Theorem~2.2]{Hutchcroft}. Our constancy theorem for $\xi(F)$ gives a stronger conclusion for the FSF.

We also remark that Lyons, Morris, and Schramm observed that the total number of ends of all components is tail measurable, hence a.s.\ constant \cite[after Question~7.8]{LMS}. Combined with Corollary~\ref{cor:components}, this gives another proof of Corollary~\ref{cor:ends} when the number of components is finite. Swap invariance cannot simply be dropped: in their example from $\mathbb Z^5$ with two rays attached, the WSF has a component with exactly two ends with probability strictly between zero and one \cite[Section~7]{LMS}.

\paragraph{Properties of all components.}
The argument for Corollary~\ref{cor:ends} also applies to properties of trees that are stable under splitting and joining. The recurrence and finite-endedness events below are also unchanged by arbitrary finite modifications. Their zero--one laws thus already follow from tail triviality \cite[Theorem~8.3]{BLPS}; we include them as examples of the general criterion.
\begin{corollary}\label{cor:hereditary}
Let $\mathsf P$ be a property of locally finite trees equipped with positive edge conductances, including finite trees, such that:
(i) if a tree $T$ has $\mathsf P$ and $e\in T$, then both components of $T\setminus\{e\}$ have $\mathsf P$; and
(ii) if disjoint trees $T_1,T_2$ have $\mathsf P$ and $e$ joins them, then $T_1\cup T_2\cup\{e\}$ has $\mathsf P$ for every positive conductance assigned to $e$.
Then the event that every component of $F$ has $\mathsf P$ has probability zero or one under FSF and under WSF. In particular, this holds for recurrence of the components, with the restricted conductances, and for having finitely many ends.
\end{corollary}

When edge conductances are bounded above, every WSF component is recurrent a.s.\ by Morris \cite[Corollary~10]{Morris}; in particular, this holds for unit conductances. For the FSF on a unimodular random rooted network, on the event that $\FSF\neq\WSF$, every component is a.s.\ transient and infinitely ended \cite[Theorem~1.12]{HN}. Corollary~\ref{cor:hereditary} gives a zero--one law on every fixed connected, locally finite, infinite network, without symmetry or a uniform bound on the conductances.

\end{document}